\documentclass[twoside,12pt,leqno]{amsart}
\usepackage{pstricks,amssymb,latexsym,verbatim,fancyhdr}
\usepackage{hyperref}
\hypersetup{citecolor=[rgb]{1,0,0,}, linkcolor=[rgb]{0,0,1}, colorlinks=true}

\theoremstyle{plain}
\newtheorem{theorem}{Theorem}
\newtheorem{lemma}[theorem]{Lemma}

\renewcommand{\leq}{\leqslant}
\renewcommand{\ge}{\geqslant}
\renewcommand{\le}{\leqslant}

\begin{document}

\title[Subgroups with maximal derived length]{\large S\MakeLowercase{ubgroups of the upper-triangular matrix group with maximal}\vskip0mm \MakeLowercase{derived length and a minimal number of generators}}

\date{8 August 1997}	

\author{S.\,P. Glasby}

\address{\small Department of Mathematics and Computing Science\\
The University of the South Pacific\\
PO Box 1168, Suva, Fiji.}

\def\F{\mathbb{F}}
\def\Z{\mathbb{Z}}
\def\GL{\text{GL}}
\def\Mat{\textup{Mat}}
\def\refB{1}
\def\refH{2}
\def\refHu{3}
\def\qed{\hfill\llap{$\sqcup$}\llap{$\sqcap$}}

\begin{abstract}
The group $U_n(\F)$ of all $n\times n$ unipotent upper-triangular matrices
over $\F$ has derived length $d:=\lceil\log_2(n)\rceil$, equivalently
$2^{d-1}<n\le 2^d$.
We prove that $U_n(\F)$ has a 3-generated subgroup of derived length~$d$,
and it has a 2-generated subgroup of derived length~$d$
if and only if $\frac{21}{32} 2^d<n\le 2^d$.
\end{abstract}

\maketitle 

\section{Introduction}

\noindent Let $\F$ be a field and let $U_n(\F)$ (or $U_n$)
denote the group of $n\times n$ upper-triangular matrices over $\F$
with 1's on the main diagonal and 0's below.
If $2^{d-1}<n\le 2^d$, then $U_n$ has derived length $d$ and has a
subgroup generated by $n-1$ elements which also has derived length $d$
(see~\cite{H}).
We show in Theorem~\ref{T:No2Gen} that $U_n$ has a 3-generated subgroup
with derived length $d$. In Theorem~\ref{T3.4} we show that $U_n$ has a
2-generated subgroup of derived length~$d$ if and only if
$\frac{21}{32} 2^d<n\le 2^d$.
It follows that the proportion, $\pi(N)$, of $n\le N$ such that $U_n$ has
a 2-generated subgroup of maximal derived length satisfies
$\frac{11}{21}<\pi(N)\le 1$, $\liminf \pi(N)=\frac{11}{21}$ and
$\limsup \pi(N)=\frac{11}{16}$.
Theorems~\ref{T:No2Gen} and \ref{T3.4} are constructive in the sense that
the generating matrices are explicitly given by recursive formulas.

We shall now introduce some notation and state some well-known properties of
$U_n$ (see~\cite{H}). The $k$th term
of the lower central series for $U_n$, denoted $\gamma_k(U_n)$, comprises
the matrices $(a_{i,j})\in U_n$ with $a_{i,j}=0$ if $0<j-i<k$.
Furthermore, the $k$th term in the derived series for $U_n$ is
$U_n^{(k)}=\gamma_{2^k}(U_n)$.

In the sequel we shall assume that $d=\lceil\log_2(n)\rceil$
and consider subgroups $G$ of $U_n$ where $G^{(d-1)}$ is not trivial.
Let $1\le i<j\le n$
and let $X_{i,j}\in U_n$ be the matrix obtained by adding row $j$
of the identity matrix, $I$, to row $i$ (so its $(i,j)$th entry is 1). Then
$$
[X_{i,j},X_{k,\ell}]=X_{i,j}^{-1}X_{k,\ell}^{-1}X_{i,j}X_{k,\ell}
$$
equals $I$ if $j<k$, and equals $X_{i,\ell}$ if $j=k$.
In order to show that $U_n$ has derived length $d$ for all $n$ satisfying
$2^{d-1}<n\le 2^d$, it suffices to show that
$\langle X_{1,2},X_{2,3},\dots,X_{n-1,n}\rangle$ has derived length $d$
when $n=2^{d-1}+1$. The latter can be proved using induction on $d$
based on the following reasoning
\begin{align*}
 X_{1,9}&=[X_{1,5},X_{5,9}]\\
	&=[\,[X_{1,3},X_{3,5}],[X_{5,7},X_{7,9}]\,]\\
	&=[\,[\,[X_{1,2},X_{2,3}],[X_{3,4},X_{4,5}]\,],
	[\,[X_{5,6},X_{6,7}],[X_{7,8},X_{8,9}]\,]\,].
\end{align*}
At the heart of this proof is a binary tree with $d$ layers and $2^d-1$
vertices. The vertices at layer $k$ are the elements
$X_{1+(i-1)2^{k-1}, 1+i2^{k-1}}$ of $U_n^{(k-1)}$.
If $j=1+(i-1)2^{k-1}$, then the vertices $X_{j,j+2^{k-1}}$ and
$X_{j+2^{k-1},j+2^k}$ of layer~$k$ are joined to $X_{j,j+2^k}$ on the
next layer.

\section{3-generated subgroups}

\noindent The idea behind the proof of Theorem~\ref{T:No2Gen} is to ``re-cycle''
vertices of the above binary tree. For example, the four matrices
$X_{1,2},X_{2,3},X_{3,4},X_{4,5}$ are not needed to show that
$X_{1,5}\in U_5^{(2)}$: three matrices suffice as
$$[\,[X_{1,2},X_{2,3}X_{3,4}], [X_{2,3}X_{3,4},X_{4,5}]\,]=
[X_{1,3}X_{1,4}, X_{3,5}]=X_{1,5}.$$
The graph at the heart of the proof of Theorem~\ref{T:No2Gen}
has fewer vertices than
the complete bipartite binary tree with $2^d-1$ vertices.
It has $d$ layers with 3 vertices per layer, where the
vertices of layer $k$ correspond to elements of $G^{(k-1)}$. Let
$A$, $B$, $C$ be the matrices corresponding to the vertices of layer~$k$.
Then the commutators $[B,C]$, $[C,A]$, $[A,B]$ correspond to the vertices of
layer~$k+1$. Thus the edges between layers $k$ and $k+1$ form a
bipartite graph $K$, and the full graph is obtained by joining
$d-1$ copies of $K$ end-to-end. Our objective is to inductively construct
three layer~1 matrices, so that at least one of the layer~$d$ matrices
is non-trivial.

Let $F$ be the free group $\langle x_1,x_2,x_3\,|\ \ \rangle$ of rank 3.
The following lemma was much harder to conceive than to prove.

\begin{lemma}\label{L2.1} Let $d$ be a positive integer, and let $n=2^{d-1}+1$.
Then there exist matrices $A_n, B_n, C_n\in U_n$ and a word
$w_n(x_1,x_2,x_3)\in F^{(d-1)}$ such that
\begin{equation}\label{E1}
  w_n(A_n,B_n,C_n)=X_{1,n},\ w_n(B_n,C_n,A_n)=I,\ \text{and}
  \ w_n(C_n,A_n,B_n)=I.
\end{equation}
\end{lemma}

\begin{proof} The proof uses induction on $d$. When $d=1$, take
$w_2(x_1,x_2,x_3)=x_1$ and $A_2=X_{1,2}$, $B_2=C_2=I$. (More generally, if
$r^3+s^3+t^3-3rst\ne 0$ in $\F$ where $r,s,t\in\Z$, then
we may take $w_2=x_1^rx_2^sx_3^t$ and find $A_2, B_2, C_2\in U_2$
such that~\eqref{E1} holds.) Suppose that $A_n, B_n, C_n\in U_n$ and
$w_n\in F^{(d-1)}$ satisfy (1). We shall construct appropriate
$A_{2n-1}, B_{2n-1}, C_{2n-1}$ and $w_{2n-1}$. Now $n=2^{d-1}+1$ and
$2n-1=2^d+1$. There is a surjective homomorphism
$$\pi\colon U_{2n-1}\to U_n\times U_n\quad\text{given by}\quad
\pi(A)=(\lambda(A),\rho(A)),$$
where $\lambda(A)$ is the upper-left $n\times n$ submatrix of $A$,
and $\rho(A)$ is the lower-right $n\times n$ submatrix of $A$. Note that
$\lambda(A)$ and $\rho(A)$ overlap at the $(n,n)$th entry of $A$, which is a 1.

Choose $A_{2n-1}, B_{2n-1}, C_{2n-1}\in U_{2n-1}$ such that
$$\pi(A_{2n-1})=(A_n,B_n),\qquad\pi(B_{2n-1})=(B_n,C_n),\qquad
\pi(C_{2n-1})=(C_n,A_n).$$
Clearly $A_{2n-1}$, $B_{2n-1}$ and $C_{2n-1}$ are not uniquely defined.
(A different choice may be obtained by multiplying by an element of
$\ker(\pi)\cong\F^{(n-1)^2}$.) Define $w_{2n-1}$ by
\[
  w_{2n-1}(x_1,x_2,x_3)=[\,w_n(x_1,x_2,x_3), w_n(x_3,x_1,x_2)\,].
\]
Clearly, $w_{2n-1}\in F^{(d)}$.

Consider $w_{2n-1}(A_{2n-1},B_{2n-1},C_{2n-1})$. Now
\begin{align*}
  \pi(\,w_n(A_{2n-1},B_{2n-1},C_{2n-1})\,)
    &=w_n(\pi(A_{2n-1}),\pi(B_{2n-1}),\pi(C_{2n-1}))\\
  &=\bigl(\,w_n(A_n,B_n,C_n),
    \,w_n(B_n,C_n,A_n)\,\bigr)\\
  &=(X_{1,n}, I).
\end{align*}
Similarly,
\begin{align*}
  \pi(w_n(B_{2n-1},C_{2n-1},A_{2n-1}))&=(I,I)\quad\text{and}\\
  \pi(w_n(C_{2n-1},A_{2n-1},B_{2n-1}))&=(I,X_{1,n}).
\end{align*}

Now $\pi(X_{1,n})=(X_{1,n},I)$ and $\pi(X_{n,2n-1})=(I,X_{1,n})$.
(Here we can tell from the context whether $X_{1,n}$ lies in
$U_{2n-1}$ or $U_n$.) Therefore
\begin{align*}
  w_n(A_{2n-1},B_{2n-1},C_{2n-1})&=X_{1,n}Z_1,\\
  w_n(B_{2n-1},C_{2n-1},A_{2n-1})&=Z_2,\quad\text{and}\\
  w_n(C_{2n-1},A_{2n-1},B_{2n-1})&=X_{n,2n-1}Z_3
\end{align*}
where $Z_1,Z_2,Z_3\in\ker(\pi)$. Since $\ker(\pi)$ is abelian, and
is centralized by both $X_{1,n}$ and $X_{n,2n-1}$, it follows that
\begin{align*}
  w_{2n-1}(A_{2n-1},B_{2n-1},C_{2n-1})&=[X_{1,n}Z_1,X_{n,2n-1}Z_3]\\
				&=[X_{1,n},X_{n,2n-1}]=X_{1,2n-1},\\
  w_{2n-1}(B_{2n-1},C_{2n-1},A_{2n-1})&=[Z_2,X_{1,n}Z_1]=I,\\
  w_{2n-1}(C_{2n-1},A_{2n-1},B_{2n-1})&=[X_{n,2n-1}Z_3,Z_2]=I.
\end{align*}
This completes the induction, and the proof.
\end{proof}

Recall the observation that $U_n(\F)$ has derived length
$d:=\lceil\log_2(n)\rceil$,
and the subgroup $\langle X_{1,2},X_{2,3},\dots,X_{n-1,n}\rangle$
has $n-1$ generators and derived length~$d$.

\begin{theorem}\label{T:No2Gen} The group $U_n(\F)$ of $n\times n$
upper-triangular matrices over a field $\F$ with all eigenvalues $1$,
has a $3$-generated subgroup whose derived length is
$d:=\lceil\log_2(n)\rceil$. Furthermore, if $n\le \frac58 2^d$
then $U_n(\F)$ has no $2$-generated subgroup of derived length~$d$.
\end{theorem}

\begin{proof}
Set $m=2^{d-1}+1$. Then both $U_m$ and $U_n$ have derived length~$d$.
By Lemma~\ref{L2.1}, $U_m$ has a 3-generated subgroup with derived
length~$d$, and hence so too does $U_n$, as $U_n$ has a subgroup
isomorphic to~$U_m$.

If $d<3$, then there are no integers in the range
$\frac12 2^d< n\le \frac58 2^d$.
Suppose $d\ge 3$ and $G=\langle A,B\rangle$ is a 2-generated subgroup of
$U_n$ where $\frac12 2^d< n\le \frac58 2^d$.
Then
\[
  \gamma_2(G)/\gamma_3(G)=\langle[A,B]\gamma_3(G)\rangle
\]
is cyclic, and therefore
$$G^{(2)}=[\gamma_2(G),\gamma_2(G)]=[\gamma_2(G),\gamma_3(G)]
\subseteq\gamma_5(G).$$
A simple induction shows $G^{(d-1)}\subseteq\gamma_{5\cdot2^{d-3}}(G)$
for $d\ge 3$. Since $n\leq 5\cdot2^{d-3}$, we have
$$G^{(d-1)}\subseteq\gamma_{5\cdot2^{d-3}}(G)
\subseteq \gamma_n(G)\subseteq\gamma_n(U_n)=\{I\}.$$
Therefore $G$ has derived length less than $d$.
\end{proof}

In the above proof, there were choices for $A_2, B_2, C_2$
and for the subsequent generators $A_n, B_n, C_n$ where $n=2^{d-1}+1$.
However, once $A_2, B_2$ and $C_2$ were specified, the $(i,i+1)$ entries
of $A_n, B_n, C_n$ ($d>1$) were determined, but the $(i,j)$ entries
with $j-i>1$ could be arbitrary. It should not surprise the reader
that different choices for $A_2, B_2$ and $C_2$ can yield
different subgroups $\langle A_n, B_n, C_n\rangle$.

We shall give an example of a 2-generated group $G=\langle A,B\rangle$
of $U_n$ that shows that both the derived length and the order can
depend on $\F$. Let $G$ be the subgroup $\langle A,B\rangle$
of~$U_6$ where $A=X_{1,2}X_{5,6}$ and $B=X_{2,3}X_{3,4}^{-1}X_{4,5}$,
and suppose that $\text{char}(\F)=p$ is prime. It follows from
$[\,[\,[B,A],B], [B,A]\,]=X_{1,6}^2$ and $[\,[\,[B,A],A], [B,A]\,]=I$
that $G$ is metabelian if $p=2$, and has derived length 3 if $p>2$.
Furthermore, $|G|=p^7$ if $p=2,3$ and $|G|=p^6$ if $p>3$. In the
latter case $G$ has maximal class (see~\cite[p.\,61]{B}).

\section{2-generated subgroups}

\noindent Suppose that $\frac58 2^d< n\le 2^d$.
It is natural to ask whether $U_n(\F)$ has a 2-generated subgroup
of derived length $d$.
If $U_m(\F)$ has a 2-generated subgroup of derived length~$d$,
then so too does $U_n(\F)$ all $n$ satisfying $m\le n\le 2^d$.
In this section we show that the smallest value of~$m$ for
which $U_m$ has a 2-generated subgroup of derived length~$d$ is
$m=\lfloor\frac{21}{32} 2^d\rfloor+1$.
This is clearly the case if $0\le d<3$. Henceforth assume that $d\ge 3$.

Let $F=\langle a,b\,|\ \ \rangle$ denote a free group of rank~2.
Then $\gamma_r(F)/\gamma_{r+1}(F)$ is an abelian group,
for each positive integer $r$, which is freely generated by the
basic commutators of weight $k$
(see~\cite{MH}). Thus a typical element of $\gamma_2(F)/\gamma_4(F)$
has the form $[b,a]^i[b,a,a]^j[b,a,b]^k\gamma_4(F)$, where
$[b,a,a]$ and $[b,a,b]$ denote left-normed commutators, i.e. $[\,[b,a],a]$
and $[\,[b,a],b]$ respectively. We shall need three lemmas
in the sequel. Lemmas~\ref{L3.1} and \ref{L3.2} are standard so we omit their proofs.

\begin{lemma}\label{L3.1}
Let $x,x'\in\gamma_r(F)$ and $y,y'\in\gamma_s(F)$ where
$x\equiv x'\mod\gamma_{r+1}(F)$ and $y\equiv y'\mod\gamma_{s+1}(F)$.
Then $[x,y]\equiv [x',y']\mod\gamma_{r+s+1}(F)$.
\end{lemma}

Applying Lemma~\ref{L3.1} to $[\,[b,a]^i[b,a,a]^j[b,a,b]^k,\,[b,a]^\ell\,]$ shows that
$$[\,[b,a,a], [b,a]\,]\gamma_6(F)
\quad\text{and}\quad[\,[b,a,b], [b,a]\,]\gamma_6(F)$$
generate $F^{(2)}\gamma_6(F)/\gamma_6(F)$.

\begin{lemma}\label{L3.2} Let $T_{r,n}(\tau_1,\dots,\tau_{n-r})$ denote a coset
of $\gamma_{r+1}(U_n)$ comprising matrices $(t_{i,j})$ satisfying
$t_{i,j}=0$ if $1\le j-i<r$, $t_{i,j}=\tau_i$ if $j-i=r$, and
$t_{i,j}$ arbitrary if $j-i>r$.
Then $[T_{r,n}(\alpha_1,\dots,\alpha_{n-r}),T_{s,n}(\beta_1,\dots,
\beta_{n-s})]$ is contained in
\[
  T_{r+s,n}(\alpha_1\beta_{1+r}-\alpha_{1+s}\beta_1,
    \dots,\alpha_{n-r-s}\beta_{n-s}-\alpha_{n-r}\beta_{n-r-s}).
\]
\end{lemma}

How might we go about finding matrices $A,B\in U_n$ such that
$\langle A,B\rangle$ has derived length~$d$? Motivated by the
previous section we suspect that the $(i,i+1)$ entries
of $A$ and $B$ are important. Let
$A\in T_{1,n}(\alpha_1,\dots,\alpha_{n-1})$ and
$B\in T_{1,n}(\beta_1,\dots,\beta_{n-1})$ where the $\alpha_i$ and the
$\beta_j$ are regarded as variables. An evaluation homomorphism from
the polynomial ring
$$P=\Z[\alpha_1,\dots,\alpha_{n-1},\beta_1,\dots,\beta_{n-1}]$$
to $\F$ gives rise to a group homomorphism
$\phi\,:\, U_n(P)\to U_n(\F)$. We shall find a word 
$c_{n-1}(a,b)\in\gamma_{n-1}(F)\cap F^{(d-1)}$ and values for the
$\alpha_i$ and $\beta_j$ in $\F$ such that
$c_{n-1}(\phi(A),\phi(B))$ equals $X_{1,n}$ or $X_{1,n}^{-1}$.

The first case not excluded by Theorem~\ref{T:No2Gen}, or already excluded,
is $n=6$. Let $c_5(a,b)$ equal $[\,[b,a,a],[b,a]\,]$. By repeated
application of Lemma~\ref{L3.2} the $(1,6)$~entry of
$c_5(A,B)$ is
\def\a{\alpha}\def\b{\beta}
\begin{align*}
  [c_5(A,B)]_{1,6}&=[\,[B,A,A],[B,A]\,]_{1,6}\\
  &=[B,A,A]_{1,4}[B,A]_{4,6}-[B,A]_{1,3}[B,A,A]_{3,6}\\
  &=-\a_1\a_2\b_3\a_4\b_5+\a_1\a_2\b_3\b_4\a_5+3\a_1\b_2\a_3\a_4\b_5
    \kern-1pt-4\a_1\b_2\a_3\b_4\a_5\\
  &\,\ \ +\a_1\b_2\b_3\a_4\a_5-2\b_1\a_2\a_3\a_4\b_5+3\b_1\a_2\a_3\b_4\a_5
    -\b_1\a_2\b_3\a_4\a_5
\end{align*}

We make some remarks about this polynomial. First each monomial summand
has five variables. The variables have distinct subscripts and contain three
$\alpha$'s and two $\beta$'s. The polynomial has integer
coefficients and $[B,A,A]$ contributes two $\alpha_i$ and one $\beta_j$
to the first three variables, or to the last three variables of each
monomial summand. Similarly, $[B,A]$ contributes an $\alpha_i$ and a $\beta_j$
to the first two variables, or to the last two variables of
each monomial summand. Thus, even without computing 
$[c_5(A,B)]_{1,6}$, we know that $\a_1\a_2\a_3\b_4\b_5$ is not a summand.
Setting $\a_1=\a_2=\b_3=\a_4=\b_5=1$ and $\b_1=\b_2=\a_3=\b_4=\a_5=0$
shows that $[c_5(\phi(A),\phi(B))]_{1,6}=-1$ and hence
$c_5(\phi(A),\phi(B))=X_{1,6}^{-1}$. This proves that $\langle\phi(A),
\phi(B)\rangle$ is a 2-generated subgroup of $U_6(\F)$ of derived length~3
for all fields $\F$.

Many of the above remarks generalize {\it mutatis mutandis} to other words in
the subgroup $\gamma_{n-1}(F)\cap F^{(d-1)}$. We shall use the following lemma
repeatedly.

\begin{lemma}\label{L3.3} {\rm (Multiplication Lemma)} With the above notation,
suppose that $w\in\gamma_r(F)$, $w'\in\gamma_s(F)$, and
$[w(A,B)]_{1,1+r}$ and $[w'(A,B)]_{1,1+s}$ have monomial summands
$m$ and $m'$ respectively.
If $r\ge s$, and no monomial summand of $[w'(A,B)]_{1,1+s}$ divides
$m$, then $m\psi_r(m')$ is a monomial summand of
$[\,[w(A,B),w'(A,B)]\,]_{1,1+r+s}$ where $\psi_r(m')$ is the polynomial
obtained from $m'$ by adding $r$ to each subscript.
\end{lemma}

\begin{proof}
By Lemma~\ref{L3.2}, $[\,[w(A,B),w'(A,B)]\,]_{1,1+r+s}$ equals
$$[w(A,B)]_{1,1+r}[w'(A,B)]_{1+r,1+r+s}-[w'(A,B)]_{1,1+s}[w(A,B)]_{1+s,1+s+r}$$
and $m\psi_r(m')$ divides the first term. However, as no monomial summand
of $[w'(A,B)]_{1,1+s}$ divides $m$, it follows that $m\psi_r(m')$
is a monomial summand of $[\,[w(A,B),w'(A,B)]\,]_{1,1+r+s}$ as desired.\qed
\end{proof}

By Theorem~\ref{T:No2Gen}, the next case of interest is when $n=11$.
Mimicking the $n=6$ case,
we seek a word $c_{10}(a,b)\in\gamma_{10}(F)\cap F^{(3)}$ such that
the polynomial $[c_{10}(A,B)]_{1,11}$ has a monomial summand
with coefficient $\pm1$. We then assign the value of 1 to the variables
in this summand, and zero to the variables not in the summand.
Since $F^{(2)}\gamma_6(F)/\gamma_6(F)$ has two generators, it follows from
Lemma~\ref{L3.1} that $F^{(3)}\gamma_{11}(F)/\gamma_{11}(F)=
\langle c_{10}(a,b)\gamma_{11}(F)\rangle$ is cyclic where
$$c_{10}(a,b)=[\,[\,[b,a,b],[b,a]\,],c_5(a,b)]
=[\,[\,[b,a,b],[b,a]\,],[\,[b,a,a],[b,a]\,]\,].$$

We abbreviate the phrase ``$m$ is a monomial summand of $p$'' by
``$m\in p$''. Now
\begin{align*}
  m_5&=\b_1\b_2\a_3\a_4\b_5\in [\,[\,[B,A,B],[B,A]\,]\,]_{1,6}\quad\text{and}\\
  m'_5&=\a_1\a_2\b_3\b_4\a_5\in [c_5(A,B)]_{1,6}.
\end{align*}
Hence by Lemma~\ref{L3.3}
$$m_{10}=m_5\psi_5(m'_5)=\b_1\b_2\a_3\a_4\b_5\a_6\a_7\b_8\b_9\a_{10}\in
[c_{10}(A,B)]_{1,11}.$$
Setting $\b_1=\b_2=\a_3=\cdots=\a_{10}=1$ and
$\a_1=\a_2=\b_3=\cdots=\b_{10}=0$ shows that $U_{11}$ has a 2-generated
subgroup of derived length~4.

\begin{theorem}\label{T3.4} Let $d=\lceil\log_2(n)\rceil$.
  Then $U_n$ has a $2$-generated subgroup of derived length~$d$
  if and only if $\frac{21}{32}2^d<n\le 2^d$.
\end{theorem}

\begin{proof} Suppose that $U_n$ has a 2-generated subgroup $G$ of
derived length~$d$. It follows from Theorem~\ref{T:No2Gen} that $\frac58 2^d<n\le 2^d$.
However, if $0\le d<5$ then $\lfloor\frac58 2^d\rfloor=
\lfloor\frac{21}{32}2^d\rfloor$. Hence $\frac{21}{32}2^d<n\le 2^d$ for
$d<5$. Suppose now that $d\ge 5$. We showed in the preamble to this
theorem that
$F^{(3)}\gamma_{11}(F)/\gamma_{11}(F)$ is cyclic. Hence by Lemma~\ref{L3.1},
$F^{(4)}\subseteq\gamma_{21}(F)$. For $d\ge 5$, a simple induction shows that
$F^{(d-1)}\subseteq\gamma_{21\cdot 2^{d-5}}(F)$. 
Since $G^{(d-1)}\subseteq\gamma_{21\cdot 2^{d-5}}(G)$ and $\gamma_n(G)=\{I\}$
it follows that $21\cdot 2^{d-5}<n\le 2^d$ as desired.

Conversely, suppose $\frac{21}{32}2^d<n\le 2^d$.
If $d=0,1,2,3,4$, then the values of $n=\lfloor\frac{21}{32}2^d\rfloor+1$ are
$1,2,3,6,11$ respectively. In each of these cases we have shown that
$U_n$ has a 2-generated subgroup of derived length~$d$.
Suppose henceforth that $d\ge 5$. We shall give a
recursive procedure for constructing a 2-generated subgroup of $U_n$.
It suffices to do this for $n=21\cdot 2^{d-5}+1$.

We use induction on $d$. The initial case when $d=5$ and $n=22$
requires the most lengthy calculations. Note that the hypothesis in Lemma~\ref{L3.3}
that no monomial summand of $[w'(A,B)]_{1,1+s}$ divides $m$ is
easily verified in the case when the first $s$ variables of $m$ have
a different number of $\a$'s than one (and hence every) summand
of $[w'(A,B)]_{1,1+s}$. A lengthy argument which
repeatedly uses this observation and the Multiplication Lemma shows that
\begin{align*}
  &m_{21}=-\a_1\a_2\a_3\b_4\b_5\a_6\psi_6(m_5)\psi_{11}(m_{10})
    \in c_{21}(a,b)\\
  &m'_{21}=\phantom{-}\a_1\a_2\b_3\b_4\b_5\a_6\psi_6(m_5)\psi_{11}(m_{10})
    \in c'_{21}(a,b)\\
  &m''_{21}=-\b_1\b_2\b_3\a_4\a_5\b_6\psi_6(m_5)\psi_{11}(m_{10})
   \in c''_{21}(a,b)
\end{align*}
where
\begin{align*}
  c_{21}(a,b)&=[\,[\,[\,[\,b,a,a,a],[b,a]\,],c_5(a,b)],c_{10}(a,b)]\\
  c'_{21}(a,b)&=[\,[\,[\,[\,b,a,a,b],[b,a]\,],c_5(a,b)],c_{10}(a,b)]\\
  c''_{21}(a,b)&=[\,[\,[\,[\,b,a,b,b],[b,a]\,],c_5(a,b)],c_{10}(a,b)].
\end{align*}

This proves the result for $d=5$ because the polynomial
$[c_{21}(A,B)]_{1,22}$ has a monomial summand with coefficient $\pm1$.
The number of $\a$'s in
$m_{21}$, $m''_{21}$, $m'_{21}$ is congruent to 0, 1, 2 modulo~3 respectively,
and so by the Multiplication Lemma
\begin{align*}
  m'_{21}\psi_{21}(m''_{21})\in d_{21}(a,b)&=[c'_{21}(a,b),c''_{21}(a,b)]\\
  m''_{21}\psi_{21}(m_{21})\in d'_{21}(a,b)&=[c''_{21}(a,b),c_{21}(a,b)]\\
  m_{21}\psi_{21}(m'_{21})\in d''_{21}(a,b)&=[c_{21}(a,b),c'_{21}(a,b)].
\end{align*}
The argument may be applied repeatedly as the number of $\a$'s
occurring in $m'_{21}\psi_{21}(m''_{21})$, $m''_{21}\psi_{21}(m_{21})$,
$m_{21}\psi_{21}(m'_{21})$ is congruent to 0, 1, 2 modulo~3 respectively.
This completes the inductive proof.\qed
\end{proof}

\section*{Acknowledgment}
I would like to thank Peter Pleasants and the referee for
their helpful comments. This research was partially supported by URC 
grant 6294-1341-70766-15.

\end{document}